\documentclass[11pt, oneside]{amsart}       
\usepackage{geometry}                        
\usepackage[parfill]{parskip}            
\usepackage{graphicx}                
                            \usepackage{amsthm}    
\newtheorem{theorem}{Theorem}[section]
\newtheorem{definition}{Definition}[section]
\newtheorem{lemma}[theorem]{Lemma}
\newtheorem{claim}[theorem]{Claim}
\newtheorem{corollary}[theorem]{Corollary}
\theoremstyle{remark}
\newtheorem*{remark}{Remark}
\usepackage{tcolorbox}

\title{Infinite Type D-modules and Higher Depth Mock Modular Forms I}
\author{EMILE BOUAZIZ}

\begin{document}
\maketitle

\begin{abstract} We study natural D-modules on the moduli stack, $\mathcal{M}_{ell}$, of elliptic curves over a field $k$ of characteristic $0$. We use this to produce an algebro-geometric version of the algebra of higher depth mock modular forms, studied from a physical perspective by numerous authors. \end{abstract}

\section{introduction} We construct a purely algebraic avatar of the algebra of (weak) \emph{higher depth mixed mock modular forms}, HDMF henceforth. These objects have attracted some attention recently as it has been observed that the $g=1$ partition functions of certain conformal field theories yield such functions, (\cite{DMZ},\cite{GMN}), in particular there is some physical speculation that the non-compact elliptic genus should be mock modular, see for example \cite{T}. A natural question, given that we know we can define the algebra of classical modular forms algebraically, is to what extent one can define HDMF algebraically. We remark that the usual definition is highly transcendental, and does not suggest a particularly obvious construction. 

We propose in this note to adopt a D-module theoretic perspective on the theory of modular forms. Vector bundles, such as the tensor powers of the Hodge bundle $\omega^{k}$ on $\mathcal{M}_{ell}$, give rise in a natural fashion to certain infinite dimensional D-modules via the \emph{jets} construction. It turns out that replacing $\omega^{k}$ with its jet D-module $\mathcal{J}\omega^{k}$, and thinking of modular forms as flat sections of the D-module $\mathcal{J}\omega^{k}$, allows for some extra flexibility. Crucial to this approach are the D-modules $\mathcal{V}^{k}$, defined as the symmetric powers of the first variation of Hodge structure of the universal elliptic curve over $\mathcal{M}_{ell}$. The algebra of algebraic (\emph{resp.} weak) HDMF, which we will denote $\mathbb{M}$, (\emph{resp.} $\mathbb{M}^{!}$), will be constructed as flat sections of D-modules inductively built from the $\mathcal{V}^{l}$ and $\mathcal{J}\omega^{k}$. The extra structure one expects, namely the so called \emph{shadow maps}, are constructed as a boundary map to a space computed (over $\mathbb{C}$) by the Eichler-Shimura isomorphism. 

A crucial point, which we now emphasize, is that the construction is suitably natural that it admits cuspidal growth conditions \emph{and} higher genera versions. We view these two points as sensible desiderata for an algebra containing the $g=1$ partition functions of suitable physical theories. 

A sketch of the theory in holomorphic and automorphic terms is also supplied, and we sketch the proofs of the comparison theorems needed. We need both a \emph{Betti to de Rham} comparison theorem and a comparison of the Betti construction with the physical incarnation of these objects. A proper account will be the subject of a sequel to this note.

\section{Acknowledgements} It was Ian Grojnowski who pointed out to me the lack of an algebro-geometric definition of mock modular forms and their variants, I thank him for this initial suggestion as well as numerous helpful conversations over the course of this work.

\section{Construction} \subsection{D-module Generalities} We assume a reasonable familiarity with the basic formalism of D-modules throughout, in particular we assume that the reader is familiar with the basic properties of the functor of derived flat sections $H^{*}_{\nabla}$. A reference for general D-module theory is \cite{Gi}. Our construction will make use of some functors relating D-modules and quasi-coherent sheaves, most significantly the jet functor $\mathcal{J}$, whose basic properties we review here, a more detailed treatment can be found in \cite{BK}. $X$ will denote a reasonable stack in what follows, $\mathbb{D}_{X}$ the category of left D-modules on $X$ and $QC(X)$ the category of quasi-coherent sheaves on $X$. \begin{definition} The functor $\mathcal{J}$, referred to as the functor of \emph{jets}, $\mathcal{J}:QC(X)\longrightarrow\mathbb{D}_{X},$ is the right adjoint to the natural functor $\mathbb{D}_{X}\longrightarrow QC(X)$.\end{definition}

There is a natural isomorphism $H^{*}_{\nabla}(X,\mathcal{J}E)\cong H^{*}(X,E)$ for any quasi-coherent sheaf $E$ on $X$. Further, there are natural maps $\mathcal{J}E\otimes\mathcal{J}F\longrightarrow\mathcal{J}(E\otimes F)$. These statements are immediate from the description of $\mathcal{J}$ as an adjoint, noting that the natural functor $\mathbb{D}_{X}\longrightarrow QC(X)$ is symmetric monoidal. \begin{remark} Recalling, for example from \cite{Lu}, the description of left D-modules as quasi-coherent sheaves on the de Rham stack, $X_{dR}$, $\mathcal{J}$ is best understood as the pushforward by the canonical map identifying nearby points, $X\rightarrow X_{dR}$. As the notation suggests, the quasi-coherent sheaf underlying the D-module $\mathcal{J}E$ is the sheaf of jets of sections of $E$. \end{remark} We will make use of another, simpler, functor as well, the functor heuristically associating to a D-module its maximal trivial sub-module, namely the functor $(-)^{fl}:\mathbb{D}_{X}\longrightarrow\mathbb{D}_{X}$  defined by $\mathcal{E}^{fl}:=H^{0}_{\nabla}(X,\mathcal{E})\otimes\mathcal{O}$.  \begin{remark} This is endowed with a natural map $\mathcal{E}^{fl}\longrightarrow\mathcal{E}$. If $X$ is a smooth variety over $\mathbb{C}$ and $\mathcal{E}$ is a vector bundle with regular singularities at $\infty$, then the Riemann-Hilbert correspondence (see \cite{Gi} for details) of Deligne encodes $\mathcal{E}$ as a representation $V$ of the topological fundamental group, $\Gamma$, of $X^{an}$. $\mathcal{E}^{fl}$ corresponds in this language to the trivial sub-object $V^{\Gamma}\longrightarrow V$. \end{remark}

\subsection{D-modules on $\mathcal{M}_{ell}$.} We will give a very brief recap of some D-modules on $\mathcal{M}_{ell}$ of interest to us. These D-modules will serve as building blocks in what follows. For a recap on the theory of the underlying quasi-coherent sheaves of interest we refer the reader to the excellent notes of R. Hain, \cite{H}.

We have natural line bundles on $\mathcal{M}_{ell}$, defined as tensor powers, $\omega^{\otimes k}$, of the Hodge bundle $\omega$, which we recall is the bundle with fibre $H^{0}(E,\Omega_{E})$ at the point $E$ of $\mathcal{M}_{ell}$. By above we obtain natural D-modules by applying the jets functor. We have also a natural rank two vector bundle with flat connection $\mathcal{V}$, which is the first variation of Hodge structure corresponding to the universal smooth elliptic curve over $\mathcal{M}_{ell}$. The fibre at $E$ is $H^{1}_{dR}(E,k)$. The underlying bundle is equipped with the \emph{Hodge filtration}, realizing it as an extension $$\omega\longrightarrow\mathcal{V}\longrightarrow\omega^{-1}.$$  The line bundles $\omega^{i}$ extend to line bundles on the compactified space $\overline{\mathcal{M}}_{ell}$, and thus so do their associated jet D-modules. The D-module $\mathcal{V}$ extends to a logarithmic D-module on $\overline{\mathcal{M}}_{ell}$, \cite{KM}.

We proceed now to the construction of the D-modules which will concern us. Note that by the above we obtain, crucially, a map of D-modules $\mathcal{V}^{k}\longrightarrow\mathcal{J}\omega^{\otimes -k}$. Note that by the functorialities of $\mathcal{J}$ discussed above we obtain further natural maps $\mathcal{V}^{l}\otimes\mathcal{J}\omega^{\otimes k}\longrightarrow\mathcal{J}\omega^{\otimes(k-l)}$, one can further use the functor $^{fl}$ to concoct more natural maps etc. 
\begin{remark} The above jet theoretic formalism allows for an elegant algebraic formulation of a lemma of Bol, we refer the reader to \cite{BH} for further details. We claim that there is a \emph{Bol exact sequence} of D-modules on $\mathcal{M}_{ell}$, $$0\longrightarrow\mathcal{V}^{k}\longrightarrow\mathcal{J}\omega^{-k}\longrightarrow\mathcal{J}\omega^{k+2}\longrightarrow 0.$$ This is not hard to prove. Taking sheaves of flat sections recovers the differential operator of Bol, $D_{Bol}:\omega^{-k}\longrightarrow\omega^{k+2}$. Further, one could apply the functor $H^{*}_{\nabla}$ to this exact sequence to obtain the familiar description of $H^{1}_{\nabla}(\mathcal{M}_{ell},\mathcal{V}^{k})$ as $M^{!}_{k+2}/ D_{Bol}M^{!}_{-k}$. One could also consider the associated exact sequence of log D-modules on the comapctification $\overline{\mathcal{M}_{ell}}$ and take flat sections. One would obtain in this manner a version of the Eichler-Shimura isomorphism. This description of the D-module $\mathcal{V}^{k}$ has the notable property that it generalises to higher genera.\end{remark}

\subsection{The Construction}

Given the above formalism, we are ready now to introduce the following, \begin{definition} For fixed $k$, we inductively define D-modules on $\mathcal{M}_{ell}$, denoted $Q^{i}_{k}$, as follows; $Q^{0}_{k}:=\mathcal{J}\omega^{\otimes k}$, and $$Q^{i+1}_{k}:=coker(\bigoplus_{l>0}\mathcal{V}^{l}\otimes (Q^{i}_{l+k})^{fl}\longrightarrow Q^{i}_{k}).$$ The corresponding logarithmic D-modules on $\overline{\mathcal{M}}_{ell}$ will be denoted $\overline{Q}^{i}_{k}$. We define $\mathbb{M}^{!,i}_{k}:=H^{0}_{\nabla}(\mathcal{M}_{ell},Q^{i}_{k})$ and $\mathbb{M}^{i}_{k}:=H^{0}_{\nabla}(\overline{\mathcal{M}}_{ell},\overline{Q}^{i}_{k})$.\end{definition} We note that there are of course natural surjections $Q^{i}_{k}\longrightarrow Q^{i+1}_{k}$, and similarly for $\overline{Q}$. In particular we obtain as well maps $\mathcal{V}^{k}\rightarrow Q^{i}_{-k}$ for $k\geq 0$. In what follows we will deal only with the D-modules on the open piece, $\mathcal{M}_{ell}$, of the moduli space of stable elliptic curves, extending the results to the cusp $\infty$ is easy.

\begin{lemma} There are natural maps $Q_{k}^{i,fl}\otimes Q_{l}^{j}\longrightarrow Q^{i+j}_{k+l}$ and $Q_{k}^{i,fl}\otimes Q_{l}^{j,fl}\longrightarrow Q^{i+j,fl}_{k+l}$, compatible with the maps $Q_{k}^{i,fl}\longrightarrow Q_{k}^{i}$. \end{lemma}
\begin{proof} An easy induction in $i$, the base case $i=0$ following from the properties of $\mathcal{J}$. \end{proof}

\begin{lemma} The maps $\mathbb{M}^{i,!}_{k}\longrightarrow\mathbb{M}^{i+1,!}_{k}$ are injective, as are the maps $\mathbb{M}^{i}_{k}\rightarrow\mathbb{M}^{i+1}_{k}$. \end{lemma} \begin{proof}We prove this over $k=\mathbb{C}$, via complex analytic methods, and use an appropriate version of the Lefschetz principle to conclude for arbitrary $k$ of characteristic $0$. Over $\mathbb{C}$, the D-modules $\mathcal{V}^{k}$ correspond to representations of the modular group $\Gamma:=SL_{2}(\mathbb{Z})$. This is a consequence of the Riemann Hilbert correspondence of Deligne, and relies crucially on the fact that each $\mathcal{V}^{k}$ has regular singularities at the cusp $\infty$. 

More specifically, they are the restrictions of the natural polynomial representations, $V^{k}$, of $SL_{2}(\mathbb{C})$. They are irreducible, and for $k\neq 0$ we have $H^{0}(\Gamma,V^{k})=0$. We record now an essentially trivial sub-lemma.\begin{claim} Let $G$ be a group, possibly infinite, and $V_{i}$ irreducible, non-trivial finite dimensional representations, indexed by possibly infinite set $I$. Let $\pi:\bigoplus_{i\in I}V_{i}\longrightarrow Q$ be a quotient. Then $H^{0}(G,Q)=0$. \end{claim}\begin{proof}\emph{(Of claim.)} $Q$ is the colimit of the directed sytem of the images of $\pi$ applied to $\bigoplus_{i\in I^{!}\subset I}V_{i}$, over the set of finite subsets $I^{!}\subset I$. The functor $H^{0}(G,-)$ commutes with directed colimits and so without loss of generality we may assume that $I$ is finite. We now induct on the cardinality of $I$, the case of $I$ a singleton being trivial. At least one of the maps $V_{i}\rightarrow Q$ must be non-zero if $Q$ is non-zero. The map is thus an inclusion by irreducibility of $V_{i}$. The inductive hypothesis implies $H^{0}(G,Q/V_{i})=0$, and the result follows from the long exact sequence and vanishing of $H^{0}(G,V_{i})$.\end{proof}

We return now to the proof of the lemma.  Recall that the Riemann-Hilbert correspondence intertwines the functors $H^{*}(\Gamma,-)$ and $H^{*}_{\nabla}(\mathcal{M}_{ell},-)$. In the language of D-modules, the above claim thus means that $H^{0}_{\nabla}$ of \emph{any} quotient of  $\bigoplus_{l>0}\mathcal{V}^{l}\otimes (Q^{i}_{l+k})^{fl}$ must vanish, since the D-modules $Q^{i,fl}_{k}$ are by definition trivial. Taking $Q$ to be the quotient by the kernel of the map $\bigoplus_{l>0}\mathcal{V}^{l}\otimes (Q^{i}_{l+k})^{fl}\longrightarrow Q^{i}_{k}$, we conclude the lemma by the long exact sequence and the inductive definition of $Q^{i+1}_{k}$. \end{proof} 

\begin{theorem} Defining $\mathbb{M}^{!}_{k}:=colim_{i}\,\mathbb{M}^{i,!}_{k}$ and $\mathbb{M}:=\bigoplus_{k}\mathbb{M}^{!}_{k}$ we obtain a filtered graded commutative algebra, whose $0^{th}$ filtered piece is the algebra of weak modular forms $M_{*}^{!}$ We have also $\mathbb{M}:=\bigoplus_{k}colim_{i}\,\mathbb{M}^{i}_{k}$, which is a filtered graded commutative subablgebra of $\mathbb{M}^{!}$, with $0^{th}$ filtered piece the graded algebra of modular forms $M_{*}$.
\end{theorem}
\begin{proof} The structure of an algebra is immediate from the above lemmas, combined with the observation that $H^{0}_{\nabla}(X,\mathcal{E}^{fl})\cong H^{0}_{\nabla}(X,\mathcal{E})$. The filtration is induced by the colimit presentation, and the grading by the weight variable $k$. That the $0$th filtered pieces are as claimed is immediate from $H^{0}_{\nabla}(X,\mathcal{J}(-))\cong H^{0}(X,-)$. \end{proof}
\subsection{Eichler-Shimura and Shadows} 
There is one more piece of structure that we should here record, namely our construction of \emph{shadow maps}, see for instance \cite{DMZ} for the depth $1$ case. The existence and properties of these maps are very easy to establish, the target space can then be understood via the theorem of Eichler-Shimura, a more detailed account of which can be found in \cite{BH}.

We will write $S_{k}$ for the vector space of weight $k$ cusp forms. We recall the theorem of Eichler-Shimura, \begin{theorem} (\emph{Eichler-Shimura}.) There is an isomorphism of $\mathbb{C}$ vector spaces, $$H^{1}(\Gamma,V_{k})\cong M_{k+2}\oplus\overline{S_{k+2}}.$$\end{theorem} 

\begin{corollary}  There is an isomorphism of $\mathbb{C}$ vector spaces $$H^{1}(\mathcal{M}_{ell},\mathcal{V}^{k})\cong M_{k+2}\oplus\overline{S_{k+2}}.$$\end{corollary}\begin{proof} The representations $V^{k}$ correspond to local systems on $\mathcal{M}^{an}_{ell}$. The de Rham realisation of these local systems are the analytified D-modules $\mathcal{V}^{an}$. The Riemann-Hilbert correspondence implies an isomorphism $H^{1}(\Gamma,V_{k})\cong H^{1}(\mathcal{M}^{an}_{ell},\mathcal{V}^{k})$ and then an appropriate version of the comparison theorem of Grothendieck implies an isomorphism $ H^{1}(\mathcal{M}^{an}_{ell},\mathcal{V}^{an,k})\cong H^{1}(\mathcal{M}_{ell},\mathcal{V}^{k})$, because the D-module $\mathcal{V}$ has regular singularities at $\infty$. \end{proof}

We define now the maps of interest, we are back to working over an arbitrary characteristic $0$ field. \begin{lemma} There exist natural \emph{refined shadow maps}, $$\xi^{i+1}_{k}:\mathbb{M}^{i+1,!}_{k}\longrightarrow\bigoplus_{l>0}H^{1}_{\nabla}(\mathcal{M}_{ell},\mathcal{V}^{l})\otimes\mathbb{M}^{i,!}_{k+l},$$ and $$\xi^{i+1}_{k}:\mathbb{M}^{i+1}_{k}\longrightarrow\bigoplus_{l>0}H^{1}_{\nabla}(\mathcal{M}_{ell},\mathcal{V}^{l})\otimes\mathbb{M}^{i}_{k+l}.$$ \end{lemma}

\begin{proof} First a technical point, the proof of claim 2.3. implies that the the inclusion of $ker(\bigoplus_{l>0}\mathcal{V}^{l}\otimes (Q^{i}_{l+k})^{fl}\longrightarrow Q^{i}_{k})$ into $Q^{i}_{k}$ is split, so we need not worry about injectivity. As such, we compose with the splitting  and define the refined shadow maps as the natural boundary maps on flat sections, coming from the definition of $Q^{i+1}_{k}$ as a quotient of $Q^{i}_{k}$. Indeed we obtain $$H^{0}_{\nabla}(\mathcal{M}_{ell},Q^{i+1}_{k})\longrightarrow H^{1}_{\nabla}(\mathcal{M}_{ell},\bigoplus_{l>0}\mathcal{V}^{l}\otimes Q^{i,fl}_{k+l}).$$ One now notes that we have an isomorphism $$H^{1}_{\nabla}(\mathcal{M}_{ell},\mathcal{V}^{l}\otimes Q^{i,fl}_{k+l})\cong H^{1}_{\nabla}(\mathcal{M}_{ell},\mathcal{V}^{k})\otimes\mathbb{M}^{i,!}_{k+l},$$ so that we are done.\end{proof}

Working now over $k=\mathbb{C}$, we use the Eichler-Shimura isomorphism and compose with the natural projection $H^{1}(\mathcal{M}_{ell},\mathcal{V}^{k})\longrightarrow\overline{S_{k+2}},$ to produce the \emph{shadow maps}.\begin{definition} Over $\mathbb{C}$, the \emph{shadow maps} are the maps $$\mathbb{M}_{k}^{i+1,!}\longrightarrow\bigoplus_{l>0}\overline{S_{l+2}}\otimes\mathbb{M}^{i,!}_{k+l},$$ and  $$\mathbb{M}_{k}^{i+1}\longrightarrow\bigoplus_{l>0}\overline{S_{l+2}}\otimes\mathbb{M}^{i}_{k+l}.$$\end{definition}

\begin{remark} These shadow maps allow us to put extra conditions on elements of $\mathbb{M}^{!}$. We can for example demand that the refined shadow lies in a particular summand. The case $l=-k$ corresponds to what have been called \emph{pure} mock modular forms of higher depth. We could also demand that the refined shadow map lie in the anti-holomorphic summand, in fact this demand is consistent with the classical/ physical definition. \end{remark}

\begin{definition} Over $\mathbb{C}$, the sub-algebra $\mathbb{M}^{an,!}\longrightarrow\mathbb{M}^{!}$ is defined by insisting that $f\in\mathbb{M}^{an,i,!}_{k}$ have purely anti-holomorphic refined shadow, ie $$\xi_{k}^{i}(f)\in H^{1}_{\nabla}(\mathcal{M}_{ell},\bigoplus_{l>0}\mathcal{V}^{l}\otimes Q^{i,fl}_{k+l})$$ lies in the anti-holomorphic summand furnished by the Eichler-Shimura isomorphism. We refer to this as the sub-algebra of \emph{analytic} (weak) HDMF. We define in a similar fashion $\mathbb{M}^{an}\longrightarrow\mathbb{M}$.\end{definition}

\begin{remark} A version of the Eichler-Shimura isomorphism over $k$ arbitrary of characteristic $0$ produces a canonical projection, $$H^{1}_{\nabla}(\mathcal{M}_{ell},\mathcal{V}^{k})\rightarrow S_{k+2}^{*},$$ so we could mimic the above definition without the $k=\mathbb{C}$ restriction if we chose, although we see no real reason to do so. Note that over $\mathbb{C}$ the Petersson inner product identifies $S_{k+2}^{*}\cong\overline{S_{k+2}}$, so the two definitions are indeed consistent. \end{remark}

The shadow maps let us determine the associated graded of $\mathbb{M}^{!}$ with respect to the depth filtration, which we will hereby denote $\mathcal{F}$. 

\begin{theorem} We have an isomorphism of bi-graded algebras $$Gr_{\mathcal{F}}\,\mathbb{M}^{!}\cong M^{!}_{*}\otimes Sym(\bigoplus_{*>0}H^{1}_{\nabla}(\mathcal{M}_{ell},\mathcal{V}^{*})),$$ where the second grading is the natural one inherited from $Sym$. \end{theorem}

\begin{proof} We have exact sequences coming from the shadow maps $$0\longrightarrow\mathbb{M}^{i,!}_{k}\longrightarrow\mathbb{M}^{i+1,!}_{k}\longrightarrow\bigoplus_{l>0}H^{1}_{\nabla}(\mathcal{M}_{ell},\mathcal{V}^{k})\otimes\mathbb{M}^{i,!}_{k+l}$$ We wish to show that these are exact on the right as well, from which the theorem will easily follow. From the definition of the shadow maps as boundary maps for the functor $H_{\nabla}^{*}$, this claim will follow from the vanishing of the groups $H^{1}_{\nabla}(\mathcal{M}_{ell},Q^{i}_{k})$, which we prove as usual by induction on the depth $i$. The case of $i=0$ reads $H^{1}(\mathcal{M}_{ell},\mathcal{J}\omega^{k})=0$. This is equivalent simply to $H^{1}(\mathcal{M}_{ell},\omega^{k})=0$, which is clear from Serre vanishing and the affineness of (the coarse space of) $\mathcal{M}_{ell}$. The inductive step now follows from vanishing of $H^{2}_{\nabla}(\mathcal{M}_{ell},\mathcal{V}^{l})$, which is clear again because the coarse space of $\mathcal{M}_{ell}$ is affine and one dimensional.\end{proof}

\begin{corollary} Over $\mathbb{C}$ there is an isomorphism $$Gr_{\mathcal{F}}\,\mathbb{M}^{an,!}\cong M^{!}_{*}\otimes Sym(\overline{S_{2-*}}).$$\end{corollary}\begin{proof} Immediate from the above lemma and the definition of $\mathbb{M}^{an,!}$. \end{proof}

 \subsection{Betti Version Sketch} We work here over $\mathbb{C}$, $\mathfrak{h}$ will denote the upper half plane and $\Gamma$ the modular group. We have an action of $\Gamma$ on $\mathcal{O}_{\mathfrak{h}}$, denoted $f\mapsto f\vert^{\gamma}_{k}$, whose fixed points are the weakly holomorphic modular forms of weight $k$. We have also the polynomial representations $V^{l}$, $l>0$, of $\Gamma$. We have inclusions $V^{l}\rightarrow (\mathcal{O}_{\mathfrak{h}},\vert^{\Gamma}_{-l})$, and maps $$(\mathcal{O}_{\mathfrak{h}},\vert^{\Gamma}_{l})\otimes(\mathcal{O}_{\mathfrak{h}},\vert^{\Gamma}_{k})\longrightarrow(\mathcal{O}_{\mathfrak{h}},\vert^{\Gamma}_{l+k}).$$  Notice the analogy with the D-module theoretic constructions above. $V^{k}$ is the Betti version of the de Rham object $\mathcal{V}^{an,k}$, under the RH-correspondence. The point is that infinite type D-module $\mathcal{J}\omega^{k}$ is supposed then to correspond to the infinite type $\Gamma$-representation $(\mathcal{O}_{\mathfrak{h}},\vert^{\Gamma}_{k})$. We can just mimic the above constructions in this Betti language now,

 \begin{definition} We define inductively define $\Gamma$-modules, $Q^{i}_{k,B}$, as follows, $Q^{0}_{k,B}:=(\mathcal{O}_{\mathfrak{h}},\vert^{\Gamma}_{k})$, and $$Q^{i+1}_{k,B}:=coker(\bigoplus_{l>0}V^{l}\otimes Q^{i,\Gamma}_{k+l,B})\longrightarrow Q^{i}_{k,B}.$$\end{definition}

Skipping through the various steps, all of which are in precise analogy with the D-module theoretic formalism above, we will eventually produce graded filtered algebras $\mathbb{M}_{B}\longrightarrow\mathbb{M}^{!}_{B}.$ Unfortunately, it would seem that the relevant cuspidal growth conditions must be inserted manually, and we will not deal with them here.

 The shadow maps allow us to define further the \emph{analytic} versions of these algebras. We sketch now the two comparison theorems of interest, details will appear in a sequel to the present note. \begin{theorem} (\emph{Betti-de Rham comparison.})  There are isomorphisms of algebras $\mathbb{M}^{!}_{B}\cong\mathbb{M}^{!}$, and $\mathbb{M}_{B}\cong\mathbb{M}$. \end{theorem} \begin{proof} (\emph{Sketch}) There is an analytic version of the jets construction, which we denote $\mathcal{J}^{an}$. It takes a coherent sheaf on an analytic space to a D-module on the analytic space. Formalism implies for algebraic $X$, and a sheaf $\mathcal{E}$ on it, we have a map $(\mathcal{J}\mathcal{E})^{an}\rightarrow\mathcal{J}^{an}\mathcal{E}^{an}$. 

Now let us note that for a smooth complex analytic space, $X^{an}$, and a vector bundle with flat connection $\mathcal{V}$ on it, we have by the Riemann-Hilbert correspondence an isomorphism $$H^{*}_{\nabla}(X^{an},\mathcal{V})\cong H^{*}(X^{an},\mathcal{V}^{\nabla}),$$ where $\mathcal{V}^{\nabla}$ denotes the (locally constant) sheaf of flat sections on $\mathcal{V}$. This is completely false for arbitrary D-modules on $X^{an}$, however we note now that there is a larger class of D-modules for which it is true. Namely, the extension-closed subcategory of $\mathbb{D}_{X^{an}}$ generated by vector bundles with flat connection as well as by the essential image of the jets functor $\mathcal{J}^{an}$. We call this the category of \emph{small} D-modules.

Now, the formalism above allows us to define analytic versions of the D-modules $Q^{i}_{j}$ via formulas identical to those of definitions 3.5 and 3.2, we denote these $Q^{i,an}_{j}$. There are analytification morphisms $H^{*}_{\nabla}(\mathcal{M}_{ell},Q^{i}_{k})\rightarrow H^{*}_{\nabla}(\mathcal{M}_{ell}^{an},Q^{i,an})$ and one checks they are isomorphisms by an easy induction, using that $\mathcal{V}^{k}$ have regular singularities yet again. One simply notes now that the analytic D-modules $Q^{i,an}_{k}$ are small, and so their cohomology is computed as sheaf cohomology of their sheaves of flat sections. Considering sheaves of vector spaces on $\mathcal{M}_{ell}^{an}$ as $\Gamma$-equivariant such on $\mathfrak{h}$, we note now that since $\mathfrak{h}$ is Stein and contractible (so that coherent sheaves have no higher cohomology and there are only trivial local systems) we have $$H^{*}_{\nabla}(\mathcal{M}_{ell}^{an},Q^{i,an}_{k})\cong H^{*}(\Gamma, Q^{i}_{k,B}),$$ whence the claim is proven for $\mathbb{M}^{!}$. The claim for $\mathbb{M}$ follows identically.

 \end{proof}

There should also be a comparison with a physicical definition, for which we refer to \cite{GMN}. The algebra appearing in \emph{loc. cit.} will be denoted $\mathbb{M}^{phys}$.

\begin{claim} (\emph{Comparison with $\mathbb{R}$-analytic definition.}) There are isomorphisms $$\mathbb{M}^{an}\cong\mathbb{M}^{phys},\mathbb{M}^{an,!}\,\cong\mathbb{M}^{phys,!}.$$\end{claim}

\begin{proof} (\emph{Sketch}.) By the above it suffices to prove that we have isomorphisms $$\mathbb{M}_{B}^{an}\cong\mathbb{M}^{phys},\,\mathbb{M}_{B}^{an,!}\cong\mathbb{M}^{phys,!}.$$ Such are furnished by the so-called \emph{non-holomorphic Eichler integral} construction, which associates an $\mathbb{R}$-analytic function $g^{*}$ to a cusp form $g\in S_{k+2}$. The real analytic function $g^{*}$ satisfies $\overline{\partial}g^{*}=y^{k}\overline{g}$, where $y:=Im(\tau)$. Further, $g^{*}$  represents a polynomial cocycle for $\Gamma$ with repsect to the weight $-k$ action on $\mathbb{R}$-analytic functions.

We now argue by induction on the filtration degree $i$. If we have $f\in \mathbb{M}^{i}_{B,k}$ with shadow $\sum_{j}\overline{g_{i}}\otimes f_{j}$, for some cusp forms $g_{j}$ and elements $f_{j}\in\mathbb{M}^{i-i}_{B}$,  then one observes that $f-\sum_{j}g_{j}^{*}f_{j}$ is an $\mathbb{R}$ analytic modular form, and is easily seen to lie in $\mathbb{M}^{phys}$, by the inductive hypothesis. This provides the isomorphism claimed.\end{proof}

\subsection{Generalisation to Higher Genera, Future Directions} We work now with moduli stacks $\mathcal{M}_{g,n}$ of smooth marked curves of genus $g$. We have analogues of the Hodge bundle $\omega(g,n)$ (now of rank $g$) and the D-module $\mathcal{V}(g,n)$, now of rank $2g$. The map $\mathcal{V}^{k}\rightarrow\mathcal{J}\omega^{-k}$ still exists, as it is deduced formally from the Hodge filtration on $\mathcal{V}$ and the functorialities of $\mathcal{J}$. We obtain thus a family of D-modules $Q^{i}_{k}(g,n)$ defined inductively as in Def. 2.2. Modulo irreducibility of the local system corresponding to $\mathcal{V}(g,n)$, which is probably known to experts, we could then prove a version of Lemma 2.2, and so produce algebras $\mathbb{M}(g,n)\longrightarrow\mathbb{M}^{!}(g,n)$. Note that we believe the physics indicates that such should exist, as one expects partition functions to exist for arbitrary genera. Further, note that in depth $0$ such objects are fairly familiar, and are referred to as \emph{Teichmueller modular forms}, see for instance \cite{C}. We would also obtain shadows of such forms, by an appropriate version of the Eichler-Shimura isomorphism in higher genera.

\end{document}